\documentclass[12pt,twoside]{amsart}
\usepackage{amsmath, amsthm, amscd, amsfonts, amssymb, graphicx}
\usepackage[bookmarksnumbered, plainpages]{hyperref}
\usepackage{enumerate}
\usepackage{mathrsfs}
\usepackage{fonttable}
\newtheorem{theorem}{Theorem}

\newtheorem{lemma}{Lemma}[section]
\newtheorem{remark}{Remark}[section]

\newtheorem{corollary}{Corollary}[section]

\newtheorem{proposition}{Proposition}[section]
\numberwithin{equation}{section}

\begin{document}
\title{Complementary inequalities to Improved AM-GM inequality}
\author{Hamid Reza Moradi$^1$ and Mohsen Erfanian Omidvar$^2$}
\subjclass[2010]{47A63, 47A30, 47A64, 15A63.}
\keywords{Operator inequalities, positive linear maps, operator norm, Kantorovich inequality, Wielandt inequality.} \maketitle
\begin{abstract}
Following an idea of Lin, we prove that if $A$ and $B$ be two positive operators such that $0<mI\le A\le m'I\le M'I\le B\le MI$, then
\begin{equation*}
{{\Phi }^{2}}\left( \frac{A+B}{2} \right)\le \frac{{{K}^{2}}\left( h \right)}{{{\left( 1+\frac{{{\left( \log \frac{M'}{m'} \right)}^{2}}}{8} \right)}^{2}}}{{\Phi }^{2}}\left( A\#B \right),
\end{equation*}
and
\begin{equation*}
{{\Phi }^{2}}\left( \frac{A+B}{2} \right)\le \frac{{{K}^{2}}\left( h \right)}{{{\left( 1+\frac{{{\left( \log \frac{M'}{m'} \right)}^{2}}}{8} \right)}^{2}}}{{\left( \Phi \left( A \right)\#\Phi \left( B \right) \right)}^{2}},
\end{equation*}
where $K\left( h \right)=\frac{{{\left( h+1 \right)}^{2}}}{4h}$ and $h=\frac{M}{m}$ and $\Phi $ is a positive unital linear map.
\end{abstract}
\pagestyle{myheadings}
\markboth{\centerline {Complementary inequalities to improved AM-GM inequality}}
{\centerline {H.R. Moradi \& M.E. Omidvar}}
\bigskip
\bigskip
\section{\bf Introduction}
\vskip 0.4 true cm
The operator norm is denoted by $\left\| \cdot \right\|$. Let $M,m$ be scalars and $I$ be the identity operator. Other capital letters are used to denote the general
elements of the ${{C}^{*}}$-algebra $\mathbb{B}\left( \mathcal{H} \right)$ of all bounded linear operators acting on a Hilbert space $\left( \mathcal{H},\left\langle \cdot,\cdot \right\rangle \right)$. We write $A\ge 0$ to mean that the operator $A$ is positive. If $A-B\ge 0$ ($A-B\le 0$), then we write $A\ge B$ ($A\le B$). A linear map $\Phi $ is positive if $\Phi \left( A \right)\ge 0$ whenever $A\ge 0$. It is said to be unital if $\Phi \left( I \right)=I$.
Let $\Phi $ be a unital positive linear map between ${{C}^{*}}$-algebras. We say that $\Phi $ is 2- positive if whenever the $2\times 2$ operator matrix $\left( \begin{matrix}
A & B \\
{{B}^{*}} & C \\
\end{matrix} \right)\ge 0$, then so is $\left( \begin{matrix}
\Phi \left( A \right) & \Phi \left( B \right) \\
\Phi \left( {{B}^{*}} \right) & \Phi \left( C \right) \\
\end{matrix} \right)\ge 0$. 
For $A,B>0$, the geometric mean $A\#B$ is defined by 
 \[A\#B={{A}^{\frac{1}{2}}}{{\left( {{A}^{-\frac{1}{2}}}B{{A}^{-\frac{1}{2}}} \right)}^{\frac{1}{2}}}{{A}^{\frac{1}{2}}}.\] 
The AM-GM inequality reads
 \[A\#B\le \frac{A+B}{2},\] 
for all positive operators $A,B$.

For an operator $A$ such that $0<mI\le A\le MI$, the following inequality is called ''Kantorovich inequality'' \cite{6}:
\[\left\langle Ax,x \right\rangle \left\langle {{A}^{-1}}x,x \right\rangle \le \frac{{{\left( M+m \right)}^{2}}}{4Mm},\qquad \text{ for }\left\| x \right\|=1.\]
Many authors investigated a lot of papers on Kantorovich inequality, among others, there is a long
research series of Mond-Pe\v{c}ari\v{c}, some of them are \cite{12,mond}. 

In this paper, by virtue of the results of \cite{012}, we obtain an improvement of Kantorovich inequality (see Theorem \ref{30}). A new refinement of operator P\'olya-Szeq\"o inequality, which can be regarded as a generalization of operator Kantorovich inequality will be introduced in Theorem \ref{32}. Theorem \ref{lin} will give precise upper bounds of {{\cite[Theorem 2.1]{03}}}. At the end, in Theorem \ref{0tb} we obtain accurate upper bound for operator Wielandt inequality, which is closely related to operator Kantorovich inequality. Our result is more extensive and precise than many previous results due to Fu and He \cite{011} and Gumus \cite{09}.
\section{\bf Refinements of Kantorovich Inequality}
\vskip 0.4 true cm
\begin{lemma}\label{14}
Let $A$ and $B$ be positive operators such that there exist the positive numbers $1<m<M$ with the property $mA\le B\le MA$. Then
\begin{equation}\label{13}
\left( 1+\frac{{{\left( \log m \right)}^{2}}}{8} \right)A\#B\le \frac{A+B}{2}.
\end{equation}
\end{lemma}
\begin{proof}
Firstly, we point out that for each $a,b>0$,
\begin{equation}\label{01}
\left( 1+\frac{{{\left( \log b-\log a \right)}^{2}}}{8} \right)\sqrt{ab}\le \frac{a+b}{2}.
\end{equation}
This inequality plays a fundamental role in our paper (for more details in this direction see \cite{012}). 

Note that if $0<ma\le b\le Ma$ with $1<m<M$, then by monotonicity of logarithm function we get
\begin{equation}\label{02}
\left( 1+\frac{{{\left( \log m \right)}^{2}}}{8} \right)\sqrt{ab}\le \frac{a+b}{2}.
\end{equation}
Taking $a=1$ in the inequality \eqref{01}, we have
\[\left( 1+\frac{{{\left( \log b \right)}^{2}}}{8} \right)\sqrt{b}\le \frac{b+1}{2}.\]
Since $mI\le {{A}^{-\frac{1}{2}}}B{{A}^{-\frac{1}{2}}}\le MI$ and $1<m<M$, on choosing $b$ with the positive operator ${{A}^{-\frac{1}{2}}}B{{A}^{-\frac{1}{2}}}$, we infer from inequality \eqref{02},
\[\left( 1+\frac{{{\left( \log m \right)}^{2}}}{8} \right){{\left( {{A}^{-\frac{1}{2}}}B{{A}^{-\frac{1}{2}}} \right)}^{\frac{1}{2}}}\le \frac{{{A}^{-\frac{1}{2}}}B{{A}^{-\frac{1}{2}}}+I}{2}.\]
Multiplying both side by ${{A}^{\frac{1}{2}}}$, we deduce the desired result \eqref{13}
\end{proof}
\vskip 0.2 true cm
As we know from \cite{010}, the following inequality is equivalent to the Kantorovich inequality:
\begin{equation}\label{28}
\left\langle Ax,x \right\rangle \left\langle Bx,x \right\rangle \le \frac{{{\left( M+m \right)}^{2}}}{4Mm}{{\left\langle A\#Bx,x \right\rangle }^{2}},
\end{equation}
where $0<mI\le A,B\le MI$ and $x\in \mathcal{H}$. 

With Lemma \ref{14} in hand, we are ready to provide a refinement of the inequality \eqref{28}.
\begin{theorem}\label{30}
Let $A,B\in \mathbb{B}\left( \mathcal{H} \right)$ such that $0<mI\le m'A\le B\le MI$ and $1<m'$. Then for every unit vector $x\in \mathcal{H}$,
\begin{equation}\label{29}
\left\langle Ax,x \right\rangle \left\langle Bx,x \right\rangle \le \frac{{{\left( M+m \right)}^{2}}}{4Mm{{\left( 1+\frac{{{\left( \log m' \right)}^{2}}}{8} \right)}^{2}}}\left\langle A\#Bx,x \right\rangle. 
\end{equation}
\end{theorem}
\begin{proof}
According to the condition $0<mI\le m'A\le B\le MI$, we can get
\[\frac{mm'}{M}I\le {{A}^{-\frac{1}{2}}}B{{A}^{-\frac{1}{2}}}\le \frac{Mm'}{m}I.\]
It follows from the above inequality that
\[\left( {{\left( {{A}^{-\frac{1}{2}}}B{{A}^{-\frac{1}{2}}} \right)}^{\frac{1}{2}}}-\sqrt{\frac{mm'}{M}}I \right)\left( \sqrt{\frac{Mm'}{m}}I-{{\left( {{A}^{-\frac{1}{2}}}B{{A}^{-\frac{1}{2}}} \right)}^{\frac{1}{2}}} \right)\ge 0,\]
and easy computations yields
\begin{equation}\label{20}
\left( \frac{\left( M+m \right)\sqrt{m'}}{\sqrt{Mm}} \right){{\left( {{A}^{-\frac{1}{2}}}B{{A}^{-\frac{1}{2}}} \right)}^{\frac{1}{2}}}\ge m'I+{{A}^{-\frac{1}{2}}}B{{A}^{-\frac{1}{2}}}.
\end{equation}
Multiplying both sides by ${{A}^{\frac{1}{2}}}$ to inequality \eqref{20} we obtain
$$\left( \frac{\left( M+m \right)\sqrt{m'}}{\sqrt{Mm}} \right)A\#B\ge m'A+B.$$
Hence for every unit vector $x$ in $\mathcal{H}$ we have
\[\left( \frac{\left( M+m \right)\sqrt{m'}}{\sqrt{Mm}} \right)\left\langle A\#Bx,x \right\rangle \ge m'\left\langle Ax,x \right\rangle +\left\langle Bx,x \right\rangle .\]
Now, by using \eqref{02} for above inequality we can find that
\[\begin{aligned}
& \left( \frac{\left( M+m \right)\sqrt{m'}}{\sqrt{Mm}} \right)\left\langle A\#Bx,x \right\rangle \\ 
&\quad \ge m'\left\langle Ax,x \right\rangle +\left\langle Bx,x \right\rangle \\ 
&\quad \ge 2\left( 1+\frac{{{\left( \log m' \right)}^{2}}}{8} \right)\sqrt{m'\left\langle Ax,x \right\rangle \left\langle Bx,x \right\rangle }. \\ 
\end{aligned}\]
Square both sides, we obtain the desired result \eqref{29}.
\end{proof}
\begin{remark}\label{16}
If we choose $B={{A}^{-1}}$ we get from Theorem \ref{30} that
\begin{equation}\label{31}
\left\langle Ax,x \right\rangle \left\langle {{A}^{-1}}x,x \right\rangle \le \frac{{{\left( M+m \right)}^{2}}}{4Mm{{\left( 1+\frac{{{\left( \log m' \right)}^{2}}}{8} \right)}^{2}}},
\end{equation}
for each $x\in \mathcal{H}$ with $\left\| x \right\|=1$. 

In this case the relation \eqref{31} represents the refinement of Kantorovich inequality.
\end{remark}
\vskip 0.2 true cm
The following reverse of H\"older-McCarthy inequality is well-known and easily proved using Kantorovich inequality:
\begin{equation}\label{100}
\left\langle {{A}^{2}}x,x \right\rangle \le \frac{{{\left( M+m \right)}^{2}}}{4Mm}{{\left\langle Ax,x \right\rangle }^{2}},\qquad \text{ for }\left\| x \right\|=1.
\end{equation} 

Applying inequality \eqref{31}, we get the following corollary that is a refinement of \eqref{100}. It can be proven by the similar method in {{\cite[Theorem 1.29]{013}}}.
\begin{corollary}\label{18}
Substituting $\frac{{{A}^{\frac{1}{2}}}x}{\left\| {{A}^{\frac{1}{2}}}x \right\|}$ for a unit vector $x$ in Remark \ref{16}, we have
\[\frac{\left\langle A{{A}^{\frac{1}{2}}}x,{{A}^{\frac{1}{2}}}x \right\rangle }{{{\left\| {{A}^{\frac{1}{2}}}x \right\|}^{2}}}\frac{\left\langle {{A}^{-1}}{{A}^{\frac{1}{2}}}x,{{A}^{\frac{1}{2}}}x \right\rangle }{{{\left\| {{A}^{\frac{1}{2}}}x \right\|}^{2}}}\le \frac{{{\left( M+m \right)}^{2}}}{4Mm{{\left( 1+\frac{{{\left( \log m' \right)}^{2}}}{8} \right)}^{2}}},\]
which is equivalent to saying that
\begin{equation}\label{19}
\left\langle {{A}^{2}}x,x \right\rangle \le \frac{{{\left( M+m \right)}^{2}}}{4Mm{{\left( 1+\frac{{{\left( \log m' \right)}^{2}}}{8} \right)}^{2}}}{{\left\langle Ax,x \right\rangle }^{2}},
\end{equation}
for each $x\in \mathcal{H}$ with $\left\| x \right\|=1$. 
\end{corollary}
\vskip 0.2 true cm
A discussion of order-preserving properties of increasing functions through the Kantorovich inequality is presented by Fujii, Izumino, Nakamoto and Seo \cite{010} in 1997. They showed that if $A,$ $B>0,$
$B\geq A$ and $0<mI \leq A\leq MI$, then
\begin{equation}\label{23}
{{A}^{2}}\le \frac{{{\left( M+m \right)}^{2}}}{4Mm}{{B}^{2}}.
\end{equation}
The following result provides an improvement of inequality \eqref{23}.
\begin{proposition}\label{17}
Let $A,B\in \mathcal{B}\left( \mathcal{H} \right)$ such that $0<mI\le m'A\le {{A}^{-1}}\le MI$, $1<m'$ and $A\le B$. Then   
\begin{equation}\label{25}
{{A}^{2}}\le \frac{{{\left( M+m \right)}^{2}}}{4Mm{{\left( 1+\frac{{{\left( \log m' \right)}^{2}}}{8} \right)}^{2}}}{{B}^{2}},
\end{equation}
\end{proposition}
\begin{proof}
For each $x\in \mathcal{H}$ with $\left\| x \right\|=1$ we have
\[\begin{aligned}
\left\langle {{A}^{2}}x,x \right\rangle &\le \frac{{{\left( M+m \right)}^{2}}}{4Mm{{\left( 1+\frac{{{\left( \log m' \right)}^{2}}}{8} \right)}^{2}}}{{\left\langle Ax,x \right\rangle }^{2}} \quad \text{(by \eqref{19})}\\ 
& \le \frac{{{\left( M+m \right)}^{2}}}{4Mm{{\left( 1+\frac{{{\left( \log m' \right)}^{2}}}{8} \right)}^{2}}}{{\left\langle Bx,x \right\rangle }^{2}} \quad \text{(since $A\le B$)}\\ 
& \le \frac{{{\left( M+m \right)}^{2}}}{4Mm{{\left( 1+\frac{{{\left( \log m' \right)}^{2}}}{8} \right)}^{2}}}\left\langle {{B}^{2}}x,x \right\rangle \quad \text{(by H\"older-McCarthy inequality)},
\end{aligned}\]
as desired.
\end{proof}
\vskip 0.2 true cm
In 1996, using the operator geometric mean, Nakamoto and Nakamura \cite{9}, proved that
\begin{equation}\label{22}
\Phi \left( A \right)\#\Phi \left( {{A}^{-1}} \right)\le \frac{M+m}{2\sqrt{Mm}},
\end{equation}
whenever $0<mI\le A\le MI$ and $\Phi $ is a normalized positive linear map on $\mathbb{B}\left( \mathcal{H} \right)$.

 It is notable that, a more general case of \eqref{22} has been studied by Moslehian et al. in {{\cite[Theorem 2.1]{20}}} which is called the operator P\'olya-Szeq\"o inequality. The operator P\'olya-Szeq\"o inequality states that: Let  $\Phi $ be a positive linear map. If $0<mI\le A,B\le MI$ for some positive real numbers $m<M$, then
\begin{equation}\label{34}
\Phi \left( A \right)\#\Phi \left( B \right)\le \frac{M+m}{2\sqrt{Mm}}\Phi \left( A\#B \right).
\end{equation}

Our second main result in this section, which is related to inequality \eqref{34} can be stated as follows:
\begin{theorem}\label{32}
Let $\Phi $ be a normalized positive linear map on $\mathbb{B}\left( \mathcal{H} \right)$ and let $A,B\in \mathbb{B}\left( \mathcal{H} \right)$ such that $0<mI\le m'A\le B\le MI$ and $1<m'$. Then 
\begin{equation}\label{15}
\Phi \left( A \right)\#\Phi \left( B \right)\le \frac{M+m}{2\sqrt{Mm}\left( 1+\frac{{{\left( \log m' \right)}^{2}}}{8} \right)}\Phi \left( A\#B \right).
\end{equation}
\end{theorem}
\begin{proof}
According to the hypothesis we get the order relation,
\[\left( \frac{\sqrt{m'}\left( M+m \right)}{\sqrt{Mm}} \right)\Phi \left( A\#B \right)\ge m'\Phi \left( A \right)+\Phi \left( B \right).\]
By using Lemma \ref{14}, we get
\[\begin{aligned}
& \left( \frac{\sqrt{m'}\left( M+m \right)}{\sqrt{Mm}} \right)\Phi \left( A\#B \right) \\ 
&\quad \ge m'\Phi \left( A \right)+\Phi \left( B \right) \\ 
&\quad \ge 2\sqrt{m'}\left( 1+\frac{{{\left( \log m' \right)}^{2}}}{8} \right)\Phi \left( A \right)\#\Phi \left( B \right). \\ 
\end{aligned}\]
Rearranging terms gives the inequality \eqref{15}.
\end{proof}
\begin{remark}
If we choose $B={{A}^{-1}}$ we get from Theorem \ref{32} that
\begin{equation}\label{33}
\Phi \left( A \right)\#\Phi \left( {{A}^{-1}} \right)\le \frac{M+m}{2\sqrt{Mm}\left( 1+\frac{{{\left( \log m' \right)}^{2}}}{8} \right)}.
\end{equation}
This is a refinement of inequality \eqref{22}.
\end{remark}
\vskip 0.2 true cm
A particular case of the inequality \eqref{33} has been known for many years: Let ${{U}_{j}}$ be contraction with $\sum\limits_{j=1}^{k}{U_{j}^{*}{{U}_{j}}}={{1}_{\mathcal{H}}}$ $\left( j=1,2,\cdots ,k \right)$. If $A$ is a positive operator on $\mathcal{H}$ satisfying $0<mI\le A\le MI$ for some scalars $m<M$, then 
\[\left( \sum\limits_{j=1}^{k}{U_{j}^{*}A{{U}_{j}}} \right)\#\left( \sum\limits_{j=1}^{k}{U_{j}^{*}{{A}^{-1}}{{U}_{j}}} \right)\le \frac{M+m}{2\sqrt{Mm}}.\] 
This inequality, proved by Mond and Pe\v{c}ari\v{c} \cite{12}, reduces to the Kantorovich inequality when $k=1$.
\begin{corollary}
By \eqref{15},
\begin{equation}\label{26}
\left( \sum\limits_{j=1}^{n}{U_{j}^{*}A{{U}_{j}}} \right)\#\left( \sum\limits_{j=1}^{n}{U_{j}^{*}{{A}^{-1}}{{U}_{j}}} \right)\le \frac{M+m}{2\sqrt{Mm}\left( 1+\frac{{{\left( \log m' \right)}^{2}}}{8} \right)}.
\end{equation}
Inequality \eqref{26} follows quite simply by noting that $\Phi \left( A \right)=\sum\limits_{j=1}^{k}{U_{j}^{*}A{{U}_{j}}}$ defines a normalized positive linear map on $\mathbb{B}\left( \mathcal{H} \right)$.
\end{corollary}

\section{\bf Some Refinements of Operator Inequalities for Positive Linear Maps}
Squaring operator inequalities has been an active area of study in the past several years; see for example, \cite{03,05,06}. The most successful one is that reverse version of the operator AM-GM inequality can be squared \cite{05}. It is surprising that Lin, {\cite[Theorem 2.1]{03}} showed that for two positive operators $A, B$ such that $0<mI\le A,B\le MI$, 
\begin{equation}\label{08}
{{\Phi }^{2}}\left( \frac{A+B}{2} \right)\le {{K}^{2}}\left( h \right){{\Phi }^{2}}\left( A\#B \right),
\end{equation}
and
\begin{equation}\label{09}
{{\Phi }^{2}}\left( \frac{A+B}{2} \right)\le {{K}^{2}}\left( h \right){{\left( \Phi \left( A \right)\#\Phi \left( B \right) \right)}^{2}},
\end{equation}
where $\Phi $ is a normalized positive linear map and $K\left( h \right)=\frac{{{\left( h+1 \right)}^{2}}}{4h}$ with $h=\frac{M}{m}$.  

In this section, we are devoted to obtain a better bound than \eqref{08} and \eqref{09}. 
\begin{theorem}\label{lin}
Let $A$ and $B$ be two positive operators such that $0<mI\le A\le m'I\le M'I\le B\le MI$. Then 
\begin{equation}\label{04}
{{\Phi }^{2}}\left( \frac{A+B}{2} \right)\le \frac{{{K}^{2}}\left( h \right)}{{{\left( 1+\frac{{{\left( \log \frac{M'}{m'} \right)}^{2}}}{8} \right)}^{2}}}{{\Phi }^{2}}\left( A\#B \right),
\end{equation}
and
\begin{equation}\label{015}
{{\Phi }^{2}}\left( \frac{A+B}{2} \right)\le \frac{{{K}^{2}}\left( h \right)}{{{\left( 1+\frac{{{\left( \log \frac{M'}{m'} \right)}^{2}}}{8} \right)}^{2}}}{{\left( \Phi \left( A \right)\#\Phi \left( B \right) \right)}^{2}},
\end{equation}
where $h=\frac{M}{m}$.
\end{theorem}
\begin{proof}
 We intend to prove
\begin{equation}\label{07}
\frac{A+B}{2}+Mm\left( 1+\frac{{{\left( \log \frac{M'}{m'} \right)}^{2}}}{8} \right){{\left( A\#B \right)}^{-1}}\le (M+m)I.
\end{equation}
According to the hypothesis we have
\[\frac{1}{2}\left( MI-A \right)\left( mI-A \right){{A}^{-1}}\le 0,\]
by easy computation we find that
\begin{equation}\label{016}
\frac{A}{2}+Mm\frac{{{A}^{-1}}}{2}\le \left( \frac{M+m}{2} \right)I
\end{equation}
and similar argument shows that
\begin{equation}\label{017}
\frac{B}{2}+Mm\frac{{{B}^{-1}}}{2}\le \left( \frac{M+m}{2} \right)I.
\end{equation}
Summing up \eqref{016} and \eqref{017}, we get
\begin{equation}\label{050}
\frac{A+B}{2}+Mm\frac{{{A}^{-1}}+{{B}^{-1}}}{2}\le \left( M+m \right)I.
\end{equation}
Whence
\[\begin{aligned}
  & \frac{A+B}{2}+Mm\left( 1+\frac{{{\left( \log \frac{M'}{m'} \right)}^{2}}}{8} \right){{\left( A\#B \right)}^{-1}} \\ 
 &\quad =\frac{A+B}{2}+Mm\left( 1+\frac{{{\left( \log \frac{M'}{m'} \right)}^{2}}}{8} \right)\left( {{A}^{-1}}\#{{B}^{-1}} \right) \\ 
 &\quad \le \frac{A+B}{2}+Mm\frac{{{A}^{-1}}+{{B}^{-1}}}{2} \quad \text{(by \eqref{13})}\\ 
 &\quad \le \left( M+m \right)I \quad \text{(by \eqref{050})}.\\ 
\end{aligned}\]
Therefore the inequality \eqref{07} is established.

Now we try to prove \eqref{04} by using the above inequality. It is not hard to see that, inequality \eqref{04} is equivalent with
\begin{equation}\label{06}
\left\| \Phi \left( \frac{A+B}{2} \right){{\Phi }^{-1}}\left( A\#B \right) \right\|\le \frac{{{\left( M+m \right)}^{2}}}{4Mm\left( 1+\frac{{{\left( \log \frac{M'}{m'} \right)}^{2}}}{8} \right)}.
\end{equation}
On the other hand, it is well known that for $A,B\ge 0$ (see {{\cite[Theorem 1]{01}}}),
\begin{equation*}
\left\| AB \right\|\le \frac{1}{4}{{\left\| A+B \right\|}^{2}}.
\end{equation*}
So, in order to prove \eqref{06} we need to show
\begin{equation}\label{05}
\Phi \left( \frac{A+B}{2} \right)+Mm\left( 1+\frac{{{\left( \log \frac{M'}{m'} \right)}^{2}}}{8} \right){{\Phi }^{-1}}\left( A\#B \right)\le (M+m)I.
\end{equation}
Besides, from the Choi's inequality {{\cite[p. 41]{02}}} we know that for any $T>0$,
\[{{\Phi }^{-1}}\left( T \right)\le \Phi \left( {{T}^{-1}} \right).\]
Therefore we prove the much stronger statement \eqref{05}, i.e.,
\begin{equation}\label{010}
\Phi \left( \frac{A+B}{2} \right)+Mm\left( 1+\frac{{{\left( \log \frac{M'}{m'} \right)}^{2}}}{8} \right)\Phi \left( {{\left( A\#B \right)}^{-1}} \right)\le (M+m)I.
\end{equation}
Using linearity of $\Phi $ and inequality \eqref{07}, we can easily obtain desired result \eqref{04}.

Remain inequality \eqref{015} can be proved analogously.
\end{proof}
\vskip 0.2 true cm
As is known to all, the Wielandt Inequality  {{\cite[p.443]{07}}} states that if $0<mI\le A\le MI$, and $x,y\in \mathcal{H}$ with $x\bot y$, then 
\[{{\left| \left\langle x,Ay \right\rangle \right|}^{2}}\le {{\left( \frac{M-m}{M+m} \right)}^{2}}\left\langle x,Ay \right\rangle \left\langle y,Ay \right\rangle .\]
In 2000, Bhatia and Davis \cite{08} proved an operator Wielandt's inequality which states that if $0<mI\le A\le MI$ and $X,Y$ are two partial isometries on $\mathcal{H}$ whose final spaces are orthogonal to each other, then for every 2-positive linear map $\Phi $ on $\mathbb{B}\left( \mathcal{H} \right)$,
\begin{equation}\label{013}
\Phi \left( {{X}^{*}}AY \right){{\Phi }^{-1}}\left( {{Y}^{*}}AY \right){{\Phi }^{-1}}\left( {{Y}^{*}}AX \right)\le {{\left( \frac{M-m}{M+m} \right)}^{2}}\Phi \left( {{X}^{*}}AX \right).
\end{equation}
Lin {{\cite[Conjecture 3.4]{05}}}, conjectured that the following assertion could be true:
\begin{equation}\label{011}
\left\| \Phi \left( {{X}^{*}}AY \right){{\Phi }^{-1}}\left( {{Y}^{*}}AY \right){{\Phi }^{-1}}\left( {{Y}^{*}}AX \right){{\Phi }^{-1}}\left( {{X}^{*}}AX \right) \right\|\le {{\left( \frac{M-m}{M+m} \right)}^{2}}.
\end{equation}
Recently, Fu and He \cite{011} attempt to solve the conjecture and get a step closer to the conjecture. But Gumus \cite{09} obtain a better upper bound to approximate the right side of \eqref{011} based on
\begin{equation}\label{012}
\left\| \Phi \left( {{X}^{*}}AY \right){{\Phi }^{-1}}\left( {{Y}^{*}}AY \right){{\Phi }^{-1}}\left( {{Y}^{*}}AX \right){{\Phi }^{-1}}\left( {{X}^{*}}AX \right) \right\|\le \frac{{{\left( M-m \right)}^{2}}}{2\sqrt{Mm}\left( M+m \right)}.
\end{equation}
The remainder of this paper presents an improvement for the operator version of Wielandt inequality. \begin{theorem}\label{0tb}
Let $0<mI\le m'{{A}^{-1}}\le A\le MI$ and $1<m'$ and let $X$ and $Y$ be two isometries such that ${{X}^{*}}Y=0$. For every 2-positive linear map $\Phi $,
$$\left\| \Phi \left( {{X}^{*}}AY \right){{\Phi }^{^{-1}}}\left( {{Y}^{*}}AY \right){{\Phi }^{^{-1}}}\left( {{Y}^{*}}AX \right){{\Phi }^{-1}}\left( {{X}^{*}}AX \right) \right\|\le \frac{{{\left( M-m \right)}^{2}}}{2\sqrt{Mm}\left( M+m \right)\left( 1+\frac{{{\left( \log m' \right)}^{2}}}{8} \right)}.$$
\end{theorem}
\begin{proof}
From \eqref{013} we have that
\[{{\left( \frac{M+m}{M-m} \right)}^{2}}\Phi \left( {{X}^{*}}AY \right){{\Phi }^{^{-1}}}\left( {{Y}^{*}}AY \right){{\Phi }^{^{-1}}}\left( {{Y}^{*}}AX \right)\le \Phi \left( {{X}^{*}}AX \right).\]
Under given assumptions we have $m{{I}\le \Phi \left( {{X}^{*}}AX \right)\le M{{I}}}$. Hence Proposition \ref{17} implies 
\[{{\left( \Phi \left( {{X}^{*}}AY \right){{\Phi }^{^{-1}}}\left( {{Y}^{*}}AY \right){{\Phi }^{^{-1}}}\left( {{Y}^{*}}AX \right) \right)}^{2}}\le \frac{{{\left( M-m \right)}^{4}}}{4Mm{{\left( M+m \right)}^{2}}{{\left( 1+\frac{{{\left( \log m' \right)}^{2}}}{8} \right)}^{2}}}{{\Phi }^{2}}\left( {{X}^{*}}AX \right).\]
Therefore
\begin{equation}\label{020}
\left\| \Phi \left( {{X}^{*}}AY \right){{\Phi }^{^{-1}}}\left( {{Y}^{*}}AY \right){{\Phi }^{^{-1}}}\left( {{Y}^{*}}AX \right){{\Phi }^{-1}}\left( {{X}^{*}}AX \right) \right\|\le \frac{{{\left( M-m \right)}^{2}}}{2\sqrt{Mm}\left( M+m \right)\left( 1+\frac{{{\left( \log m' \right)}^{2}}}{8} \right)},
\end{equation}
which completes the proof. 
\end{proof}
Based on inequality \eqref{020}, we obtain a refinement of inequality \eqref{012}.

\bibliographystyle{alpha}

\vskip 0.4 true cm

\tiny$^1$Department of Mathematics, Mashhad Branch, Islamic Azad University, Mashhad, Iran.

{\it E-mail address:} hrmoradi@mshdiau.ac.ir
\vskip 0.4 true cm

$^2$Department of Mathematics, Mashhad Branch, Islamic Azad University, Mashhad, Iran.

{\it E-mail address:} erfanian@mshdiau.ac.ir
\end{document}